\documentclass[11pt,a4paper]{article}
\usepackage[T1]{fontenc}
\usepackage[utf8]{inputenc}
\usepackage{lmodern,microtype}
\usepackage[margin=25mm]{geometry}
\usepackage{amsmath,amssymb,amsthm,mathtools,booktabs,array,longtable}
\usepackage{enumitem}
\usepackage[hidelinks]{hyperref}
\usepackage{xurl}
\hypersetup{
  pdftitle={The maximal determinant of order 15 over the third roots of unity},
  pdfauthor={Daoyu Dong},
  pdfsubject={Maximal determinants over third roots of unity},
  pdfkeywords={maximal determinant, roots of unity, Gram matrix, computer-assisted proof}
}
\usepackage{aliascnt}
\usepackage[nameinlink,noabbrev,capitalise]{cleveref}
\usepackage{fancyhdr}
\newtheorem{theorem}{Theorem}[section]
\newaliascnt{lemma}{theorem}
\newtheorem{lemma}[lemma]{Lemma}\aliascntresetthe{lemma}
\newaliascnt{proposition}{theorem}
\newtheorem{proposition}[proposition]{Proposition}\aliascntresetthe{proposition}
\newaliascnt{corollary}{theorem}
\newtheorem{corollary}[corollary]{Corollary}\aliascntresetthe{corollary}
\theoremstyle{remark}
\newaliascnt{remark}{theorem}
\aliascntresetthe{remark}
\crefname{theorem}{Theorem}{Theorems}
\crefname{lemma}{Lemma}{Lemmas}
\crefname{proposition}{Proposition}{Propositions}
\crefname{corollary}{Corollary}{Corollaries}
\crefname{remark}{Remark}{Remarks}
\DeclareMathOperator{\tr}{tr}\DeclareMathOperator{\rank}{rank}
\DeclareMathOperator{\diag}{diag}\DeclareMathOperator{\Aut}{Aut}
\DeclareMathOperator{\GL}{GL}
\newcommand{\Eis}{\mathbb Z[\omega]}
\newcommand{\muT}{\mu_3}
\newcommand{\N}{\operatorname N}
\newcommand{\Qset}{\mathcal Q}
\newcommand{\one}{\mathbf1}
\newcommand{\Bval}{277868041444786176}
\newcommand{\file}[1]{\path{#1}}
\newcommand{\U}{\mathcal U}
\setlist{nosep}
\title{The maximal determinant of order 15\\over the third roots of unity}
\author{Daoyu Dong\\[4pt]
  \small University of Electronic Science and Technology of China\\
  \small Chengdu, China}
\date{}
\begin{document}
\maketitle
\begin{abstract}
We prove that the largest absolute determinant of a $15\times15$ matrix with entries in $\{1,\omega,\omega^2\}$, where $\omega^3=1$ and $\omega\ne1$, is $120932352\sqrt{19}$. A matrix attaining this value was already known; the issue is maximality. The proof follows the Gram-matrix search method used in the real maximal-determinant problem. A sharp determinant bound in terms of the off-diagonal Gram energy, combined with Eisenstein congruences and a published classification of difference matrices, reduces a strict improvement to 23 possible energies. Componentwise Schur bounds and exact support catalogues then leave a small number of cases, excluded by orthogonality, rank, and spectral-intertwining conditions. We present the finite tests in a common framework and supply their replay programs. For the displayed extremizer we give a dephased representative, its Gram matrix and spectrum, and an exact computation of its monomial automorphism group, which has order 1008. The proof determines the value of the maximum, but does not classify all equality cases.
\end{abstract}

\section{Introduction}
Let $\omega=\exp(2\pi i/3)$ and $\muT=\{1,\omega,\omega^2\}$. We consider
\[
 D_3(n)=\max\{|\det H|:H\in\muT^{n\times n}\}.
\]
Hadamard's inequality gives $D_3(n)\le n^{n/2}$, with equality precisely when $HH^*=nI$ \cite{Hadamard,Butson}. At order 15 the equality case is impossible. For example, $\det H\in\Eis$, whereas $15^{15}$ is not an Eisenstein norm: the inert prime 5 occurs to an odd exponent. This is a global obstruction, not a pairwise one. Two length-15 third-root vectors can be orthogonal, since their coordinatewise quotients can contain each third root five times.

The record value listed by Nu\~nez Ponasso at this order is $120932352\sqrt{19}$. Nu\~nez Ponasso's table records the corresponding squared determinant, after division by $3^{14}$, as $2^{22}3^6 19$, without a maximality proof \cite{Nunez}. We prove that this value is optimal. The construction of a matrix with this determinant is not a contribution of the present paper.

\begin{theorem}\label{thm:main}
For every $H\in\muT^{15\times15}$,
\[
 |\det H|^2\le B:=2^{22}3^{20}19=\Bval.
\]
Equality is attained by the matrix in \cref{app:matrix}. Consequently,
\[
 D_3(15)=2^{11}3^{10}\sqrt{19}=120932352\sqrt{19},
 \qquad \frac{D_3(15)}{15^{15/2}}=0.7965900126\ldots.
\]
\end{theorem}

\subsection{The real problem and the Gram-matrix method}
The real maximal-determinant problem provides both a precedent and a warning. For matrices with entries $\pm1$, the order modulo 4 strongly constrains the off-diagonal Gram entries. Barba's bound for odd $n$ is $\sqrt{2n-1}(n-1)^{(n-1)/2}$. For $n\equiv2\pmod4$, the Ehlich--Wojtas bound is $(2n-2)(n-2)^{(n-2)/2}$. Ehlich's treatment of $n\equiv3\pmod4$ uses a more detailed block structure. These bounds, their equality conditions, and their limitations are reviewed in \cite{Survey}.

When a general bound does not settle an order, one can first enumerate possible positive definite Gram matrices and then test whether they factor over the prescribed alphabet. This strategy is already present in the work of Moyssiadis and Kounias on order 17 \cite{MK82}, and in the subsequent order-21 calculation of Chadjipantelis, Kounias and Moyssiadis \cite{CKM87}. Orrick used a Gram-candidate and decomposition search to settle the real order-15 problem \cite{Orrick}; related methods were later used at orders 19 and 37 \cite{Brent}. We credit this line of work for the general search strategy used here.

Several parts of that strategy transfer without change to complex matrices: positivity, Schur complements, eigenvalue bounds, and the identity $(HH^*)H=H(H^*H)$. The arithmetic does not. Real parity conditions are replaced here by congruences in $\Eis$, and a squared determinant need only be an Eisenstein norm, not an integer square. Moreover, the vanishing of a third-root inner product imposes balanced differences, a stronger condition than ordinary Hamming equidistance. Results for other complex alphabets, such as fourth-root designs \cite{Cohn}, cannot simply be substituted for third-root results.

\subsection{Organization of the search}
Write
\begin{equation}\label{eq:energy}
 G=HH^*,\qquad K=H^*H,\qquad
 Q=\sum_{1\le i<j\le15}\N(G_{ij})
   =\frac{\tr(G^2)-3375}{2},
\end{equation}
where $\N(z)=z\bar z$. The matrices $G$ and $K$ have the same eigenvalues and the same energy $Q$. Thus $2Q$ is the squared Frobenius norm of the off-diagonal part of either Gram matrix.

The proof starts with a hypothetical strict improvement $\det G>B$. The determinant envelope in \cref{thm:envelope} gives $Q\le168$. Congruence and coding constraints then leave the finite set $\Qset$ in \cref{cor:shells}. For each energy, only finitely many off-diagonal entries are possible: every entry is a sum of 15 third roots. We enumerate supports and their admissible labels, apply upper bounds before evaluating determinants, and test the surviving Grams for necessary properties of a third-root factorization. A positive definite Gram with determinant above $B$ is not itself a counterexample; it must also be realizable as $HH^*$.

The classification of difference matrices by Lampio and \"Osterg\aa rd \cite{LO} is particularly useful. It says that the maximum number of mutually orthogonal third-root rows of length 15 is \emph{nine}. This is stronger than the upper bound 12 obtained from ordinary ternary equidistant codes \cite{TB}. It removes several configurations that otherwise require separate decomposition arguments. The use of this external classification is specified in \cref{thm:orthogonality,sec:repro}.

\section{Arithmetic, orthogonality, and the finite energy set}
\subsection{Eisenstein arithmetic and row colors}
For $z=a+b\omega$,
\[
 \bar z=(a-b)-b\omega,\qquad \N(z)=a^2-ab+b^2.
\]
The Euclidean ring $\Eis$ has six units, $\{\pm1,\pm\omega,\pm\omega^2\}$. A positive integer is an Eisenstein norm if and only if every prime congruent to 2 modulo 3 has even valuation. We use this criterion as an exact sieve; it does not assert that every norm is the squared determinant of a third-root matrix. Subtracting one row of $H$ from the others also shows that $3^{14}$ divides $|\det H|^2$.

Write $H_{ij}=\omega^{e_{ij}}$ and assign row $i$ the color
$r_i=\sum_j e_{ij}\pmod3$. Put $\delta=1-\omega$, so $\N(\delta)=3$ and $\delta^2=-3\omega$.
\begin{lemma}\label{lem:colors}
For $i\ne j$,
\[
 G_{ij}\equiv(r_j-r_i)\delta\pmod{3\Eis}.
\]
A nonzero same-color inner product has norm at least 9, and its norm is divisible by 9. A different-color inner product has norm congruent to 3 modulo 9 and is never zero. If the color multiplicities are $a,b,c$, then
\begin{equation}\label{eq:color-energy}
 Q\ge3(ab+ac+bc),\qquad Q\equiv3(ab+ac+bc)\pmod9.
\end{equation}
In particular, $Q\equiv0$ or $6\pmod9$.
\end{lemma}
\begin{proof}
Modulo $3\Eis$, $\omega^e=(1-\delta)^e\equiv1-e\delta$. Summing the 15 coordinatewise quotients gives the first assertion. If $z\equiv0\pmod3$, its norm is divisible by 9. If $z\equiv\pm\delta\pmod3$, direct expansion of $\N(z)$ gives $\N(z)\equiv3\pmod9$. Finally, $a+b+c=15$ implies $ab+ac+bc\equiv0$ or 2 modulo 3.
\end{proof}

The complete entry catalogue is obtained without search over matrices. If a row difference contains $n_0,n_1,n_2$ occurrences of $0,1,2$, then
\begin{equation}\label{eq:entry-catalogue}
 z=(n_0-n_2)+(n_1-n_2)\omega,
 \qquad n_0+n_1+n_2=15,\quad n_i\ge0.
\end{equation}
There are $\binom{17}{2}=136$ triples to examine. This catalogue supplies the allowed norms and associate classes to every subsequent enumeration.

\subsection{The exact orthogonality parameter}
Let $M_p(n)$ denote the maximum number of mutually orthogonal rows in $\mu_p^n$, where $p$ is prime. Let $B_p(n,d)$ denote the maximum size of a $p$-ary code in which every two distinct words have Hamming distance exactly $d$.
\begin{proposition}\label{prop:prime-code}
Let $p$ be prime. Two rows of $p$th roots of unity are orthogonal if and only if each element of $\mathbb Z_p$ occurs equally often among their exponent differences. Consequently, $M_p(n)=1$ if $p\nmid n$; when $p\mid n$,
\[
 M_p(n)\le B_p\bigl(n,n(p-1)/p\bigr).
\]
More precisely, $M_p(n)$ equals the maximum number of rows in a difference matrix with $n$ columns over $\mathbb Z_p$.
\end{proposition}
\begin{proof}
Let $n_a$ count the coordinate differences equal to $a$, and let $\zeta$ be a primitive $p$th root. Orthogonality is $f(\zeta)=0$ for $f(x)=\sum_{a=0}^{p-1}n_ax^a$. The minimal polynomial of $\zeta$ is $1+x+\cdots+x^{p-1}$, of the same degree as the largest possible degree of $f$. Thus all coefficients $n_a$ are equal. The number of nonzero differences is then $n(p-1)/p$. Applying this to every pair of rows gives both conclusions.
\end{proof}

\begin{theorem}[Published orthogonality bounds at length 15]\label{thm:orthogonality}
The ordinary ternary equidistant-code parameter satisfies $B_3(15,10)=12$ \cite{TB}. The stronger, balanced-difference parameter satisfies
\[
 M_3(15)=9
\]
by the difference-matrix classification of Lampio and \"Osterg\aa rd \cite{LO}.
\end{theorem}
The equality in the second statement, including existence of nine rows, is an external classification result. We do not derive it from the first statement or rerun that classification. An equidistant code of size 12 need not have balanced nonzero differences, so the first equality does \emph{not} imply $M_3(15)=12$.

For a Gram matrix, let its support graph join $i$ and $j$ when $G_{ij}\ne0$. Every independent set is a set of mutually orthogonal rows. Hence every realizable principal submatrix has support independence number at most nine. In particular, a same-color class of size $a$ requires at least $\max(0,a-9)$ nonzero pairs: deleting one endpoint of each of $m$ edges leaves an independent set of size at least $a-m$.

\subsection{A sharp trace--variance determinant bound}
The following elementary positive-definite-matrix result is the analytic basis of the energy search. Its statement is independent of the alphabet.
\begin{theorem}[Sharp determinant envelope]\label{thm:envelope}
Let $A$ be positive semidefinite of order $m\ge2$, with $\tr A=m\mu$, $\mu>0$. Set
\[
 V=\tr(A^2)-m\mu^2,\qquad
 t=\sqrt{\frac{V}{m(m-1)}}.
\]
Then $0\le t\le\mu$ and
\begin{equation}\label{eq:envelope}
 \det A\le (\mu-t)^{m-1}\bigl(\mu+(m-1)t\bigr).
\end{equation}
For $0<t<\mu$, equality holds exactly when the eigenvalues are $\mu-t$ with multiplicity $m-1$ and $\mu+(m-1)t$ with multiplicity one. The bound is also sharp among matrices with constant diagonal $\mu$.
\end{theorem}
\begin{proof}
Optimize the product of the nonnegative eigenvalues subject to their first two moments. The feasible set is compact. For $0<V<m(m-1)\mu^2$, it contains the positive tuple stated in the theorem, so a maximizing tuple is interior. Lagrange multipliers for the logarithm of the product give
\[
 \frac1{x_i}=\alpha+2\beta x_i.
\]
There are at most two positive coordinate values, say $a<b$. Subtracting their equations gives $2\beta=-1/(ab)$. If $b$ occurred twice, the tangent vector with entries $1,-1$ in those two coordinates would preserve both constraints to first order. The Hessian of the Lagrangian on that vector would be positive, since its diagonal entry there is $-1/b^2+1/(ab)>0$, contrary to the second-order necessary condition for a maximum. Thus $b$ occurs once. Solving the two moment equations gives $a=\mu-t$ and $b=\mu+(m-1)t$.

The cases $V=0$ and $V=m(m-1)\mu^2$ are respectively scalar and rank-one spectra. For sharpness with prescribed diagonal, take
$A=(\mu-t)I+mtuu^*$ with $|u_i|=m^{-1/2}$.
\end{proof}
For a block of order $m$ with diagonal 15 and energy $q$, write
\[
 \U_m(q)=\left(15-\sqrt{\frac{2q}{m(m-1)}}\right)^{m-1}
 \left(15+(m-1)\sqrt{\frac{2q}{m(m-1)}}\right).
\]
In particular, $\det G\le\U_{15}(Q)$, where the square root is $\sqrt{Q/105}$. The right side decreases strictly with $Q>0$. The tail cutoff is exact: with $t_0=319/250$,
\begin{equation}\label{eq:tail}
 t_0^2<171/105,\qquad
 (15-t_0)^{14}(15+14t_0)<B.
\end{equation}
Both comparisons are rational integer comparisons after clearing denominators. Therefore a strict improvement has $Q<171$, and by \cref{lem:colors} it has $Q\le168$.

\begin{corollary}[Finite energy set]\label{cor:shells}
If $\det G>B$, then
\begin{equation}\label{eq:qset}
\begin{split}
 Q\in\Qset=\{&54,63,72,81,87,90,96,99,105,108,114,117,\\
             &123,126,132,135,141,144,150,153,159,162,168\}.
\end{split}
\end{equation}
Using only $B_3(15,10)=12$ instead of the sharp difference-matrix theorem gives the weaker 29-element set
\[
 \Qset\cup\{27,36,45,60,69,78\}.
\]
Using no coding bound gives the 38 integers $0\le Q\le168$ congruent to 0 or 6 modulo 9.
\end{corollary}
\begin{proof}
A pure-color support requires at least six edges, and hence $Q\ge54$. For a mixed coloring with largest class $a\ge10$,
\[
 Q\ge3a(15-a)+9(a-9)\ge87 \qquad(10\le a\le14).
\]
If every class has size at most nine, the cross-pair count is at least 54, giving $Q\ge162$. Combining these bounds with \cref{lem:colors,thm:envelope} gives \eqref{eq:qset}. With orthogonality bound 12, the corresponding minima are 27 for a pure coloring and 60 for a mixed coloring. Equivalently, all claims follow by listing $a\ge b\ge c\ge0$, $a+b+c=15$, and applying
\[
 Q\ge3(ab+ac+bc)+9\sum_{x\in\{a,b,c\}}\max(0,x-M),
 \quad Q\equiv3(ab+ac+bc)\pmod9
\]
with $M=9$ or $M=12$.
\end{proof}

\section{Tools for the finite Gram search}\label{sec:tools}
All bounds below are necessary conditions. We may enlarge a candidate family when doing so makes enumeration simpler; an exclusion of the enlarged family remains valid. We never identify row and column Gram supports without justification.

\subsection{Reduction of the color partitions}\label{sec:partitions}
\begin{lemma}\label{lem:partitions}
In a strict improvement, the row and column color partitions each belong to
\[
 (15),\quad(14,1),\quad(13,2),\quad(13,1,1).
\]
\end{lemma}
\begin{proof}
Suppose a same-color class of size $a$ is internally orthogonal. Its Gram block is $15I_a$. The Schur complement on the other $15-a$ coordinates has trace at most $(15-a)(15-a/5)$, because the cross energy is at least $3a(15-a)$. Thus
\[
 \det G\le15^a(15-a/5)^{15-a}<B\qquad(3\le a\le12).
\]
The displayed inequalities are rational. Consequently, each color class of size between 3 and 12 has internal energy at least 9 in a strict improvement. Listing the partitions subject to this fact and $Q\le168$ leaves the four stated partitions and
\[
 (12,3),\ (12,2,1),\ (11,4),\ (11,3,1),\ (11,2,2),\ (10,5).
\]
For the latter six, use a largest-class block $A$ of size $a$, internal energy $e$, and eigenvalue upper bound $L$. If its cross energy is at least $c$, Schur complementation and the arithmetic--geometric mean inequality give
\begin{equation}\label{eq:partition-bound}
 \det G\le\U_a(e)\left(15-\frac{c}{L(15-a)}\right)^{15-a}.
\end{equation}
The complete list of parameters needed is in \cref{tab:partition}. Every resulting bound is less than $B$. The small-energy eigenvalue caps follow by listing the supports with at most four edges; for the remaining cases the trace--variance eigenvalue bound suffices. The entry catalogue and this finite list are regenerated by \file{proof/order15_mu3_unified_audit/programs/verify_foundations.py}.
\end{proof}
\begin{table}[ht]
\centering
\caption{Parameters in \eqref{eq:partition-bound}; entries in each row are paired in order.}\label{tab:partition}
\begin{tabular}{@{}cccc@{}}\toprule
$a$ & $e$ & $L$ & $c$\\\midrule
10 & $9$ & $18$ &150\\
11 & $9,18,27$ & $20,21,23$ &132\\
12 & $9,18,27,36,45,54$ & $18,77/4,21,22,241/10,25$ &108\\\bottomrule
\end{tabular}
\end{table}

\subsection{Componentwise Schur bounds}
Consider a majority-color block $A=\bigoplus_i A_i$, where component $i$ has $v_i$ vertices and internal energy $9u_i$. Let $D_i\ge\det A_i$ and $L_i\ge\lambda_{\max}(A_i)$. Isolated vertices are included as components with $(v_i,u_i,D_i,L_i)=(1,0,15,15)$.

\begin{lemma}\label{lem:schur}
For a partition $(14,1)$ with total cross energy $c$, set
\[
 s_{14}=\sum_i\frac{3v_i}{L_i}+\frac{c-42}{\max_iL_i}.
\]
Then
\begin{equation}\label{eq:schur14}
 \det G\le\left(\prod_iD_i\right)\max\{0,15-s_{14}\}.
\end{equation}
For a size-13 majority block and two remaining coordinates, let $c$ be the energy from the majority to those coordinates and let $h$ be the norm of their mutual inner product. Put
\[
 s_{13}=\sum_i\frac{6v_i}{L_i}+\frac{c-78}{\max_iL_i}.
\]
If $s_{13}<30$, then
\begin{equation}\label{eq:schur13}
 \det G\le\left(\prod_iD_i\right)
 \left[(15-s_{13}/2)^2-
       \max\{0,\sqrt h-s_{13}/2\}^2\right].
\end{equation}
If $s_{13}\ge30$, no positive definite candidate is possible.
\end{lemma}
\begin{proof}
For one remaining coordinate, write $G=\left(\begin{smallmatrix}A&z\\z^*&15\end{smallmatrix}\right)$. Each coordinate of $z$ has squared modulus at least 3. Therefore
$z^*A^{-1}z\ge\sum_i\|z_i\|^2/L_i\ge s_{14}$, proving \eqref{eq:schur14}.

For two coordinates, write $U^*A^{-1}U=\left(\begin{smallmatrix}x&z\\\bar z&y\end{smallmatrix}\right)$ and $s=x+y$. Then $s\ge s_{13}$ and $|z|\le\sqrt{xy}\le s/2$. The determinant of the remaining Schur block is at most
\[
 (15-s/2)^2-\max\{0,\sqrt h-s/2\}^2.
\]
This expression is decreasing for $0\le s\le30$, since $\sqrt h\le15$ for a positive semidefinite Gram. Its substitution at $s_{13}$ gives the result. A positive definite Schur block must have positive trace $30-s$.
\end{proof}
These estimates retain the loss coming from every isolate. Discarding isolates, or assuming that one can choose the orientation so that they disappear, would invalidate the coverage of several boundary cases.

\subsection{Component determinant bounds}\label{sec:moments}
For each component we take the least of the applicable bounds: the sharp envelope, an exact small-support determinant maximum, and a fourth-moment refinement. For a tree with unit edge norms after division by 3, phases can be switched away for the purpose of computing its spectrum. All nonisomorphic trees of the required order are listed. For connected supports with at most seven vertices and energy at most ten units, all graph-atlas supports, norm assignments, and chord phases are evaluated exactly.

For larger components the following bound avoids an unnecessarily large phase enumeration.
\begin{lemma}\label{lem:moment}
Let $A=15I_v+E\succ0$, with $E$ Hermitian, zero diagonal, $\tr(E^2)=18u$, $\rank E\le d$, and $\lambda_{\max}(A)\le L$. If
\[
 W\ge\sum_{\{i,j,k\}}\frac{|E_{ij}E_{jk}E_{ki}|}{27},
\]
where the sum is over support triangles, then
\begin{equation}\label{eq:moment}
 \log\frac{\det A}{15^v}
 \le-\frac u{25}+\frac{2W}{125}-\frac{3u^2}{125Ld}.
\end{equation}
\end{lemma}
\begin{proof}
For $-1<x\le\rho$, $\rho\ge0$, integration of
$1/(1+x)=1-x+x^2-x^3/(1+x)$ gives
\[
 \log(1+x)\le x-x^2/2+x^3/3-\frac{x^4}{4(1+\rho)}.
\]
Apply this to the eigenvalues of $E/15$, with $\rho=(L-15)/15$. Their first moment is zero. Also $\tr E^3\le162W$ and $\tr E^4\ge(18u)^2/d$. Substitution gives \eqref{eq:moment}.
\end{proof}
A valid choice of eigenvalue cap is
\[
 L=15+\sqrt{18u(v-1)/v}.
\]
If the support has clique number at most $k$, the sharper cap $15+\sqrt{18u(k-1)/k}$ is valid. Indeed, Cauchy--Schwarz bounds $|x^*Ex|^2$ by $36u\sum_{ij\in E(\Gamma)}|x_i|^2|x_j|^2$. For nonnegative weights summing to one, the latter edge sum is at most $(k-1)/(2k)$: repeatedly move the mass between a nonadjacent pair onto one vertex without decreasing the sum, until the support of the weights is a clique, then use Cauchy--Schwarz again.

The triangle bound is also finite and explicit. For a connected graph with $m$ edges and $v$ vertices, put $g=m-v+1$. The number of triangles is bounded using its cycle rank and biconnected blocks. Deleting a minimum-degree vertex of degree $d$ from a 2-connected block reduces the cycle rank to $g-d+1$ and removes at most $\binom d2$ triangles. Together with additivity over blocks, this gives a recursion for an upper bound $T(g)$. The code also takes the minimum with $\binom v3$, $m(m-1)/6$, $g(g+1)/2$, and $2^{g-1}$. The last bound counts the vectors of odd edge parity in the cycle space: there are either none or exactly $2^{g-1}$, and every triangle gives one. For a norm assignment $w_e=|E_e|^2/9$, put $\Delta_e=\sqrt{w_e}-1$. Expanding each triangle weight gives the valid upper bound
\[
 W\le T(g)+g\sum_e\Delta_e+
       \sum_{e<f}\Delta_e\Delta_f+
       \sum_{e<f<h}\Delta_e\Delta_f\Delta_h.
\]
An edge belongs to at most $g$ triangles, and two edges to at most one, which justifies the coefficients. The alternative $W\le T(g)\sqrt{w_{(1)}w_{(2)}w_{(3)}}$, using the three largest weights, is also applied. For each weight partition the smaller of these two bounds is used; maximizing this value over all allowed weight partitions supplies a bound depending only on $(v,u)$.

All square roots in these comparisons are enclosed by rational bounds with the direction of rounding chosen explicitly. When \eqref{eq:moment} gives $\det A\le15^v e^{-a}$ with $a>0$, the program uses
\[
 e^{-a}\le\left(\sum_{j=0}^{24}\frac{a^j}{j!}\right)^{-1}.
\]
No decimal approximation is used to discard a candidate. If a rank or eigenvalue cap is known, the moment optimization is repeated with that cap: enumerate the number of eigenvalues at the cap and the two possible remaining stationary values. Boundary points with determinant zero need not be retained.

\subsection{Two-sided rank and spectral restrictions}
For pure-color Grams put $T=(G-15I)/3$ and $S=(K-15I)/3$. They satisfy
\begin{equation}\label{eq:intertwine}
 TH=HS.
\end{equation}
\begin{lemma}[Isolates and kernel projections]\label{lem:projection}
If the support of $T$ has $r$ isolates, then
\[
 15(P_{\ker S})_{jj}\ge r\quad\text{for every }j.
\]
The analogous statement holds with $T$ and $S$ exchanged. In particular, an isolate on one side forbids a support leaf on the other side.
\end{lemma}
\begin{proof}
The $r$ corresponding rows of $H$ are mutually orthogonal, each has squared norm 15, and by \eqref{eq:intertwine} lies in $\ker S$. If they form $R$, then $R^*R/15\preceq P_{\ker S}$. Every diagonal entry of $R^*R$ equals $r$. At a support leaf, the kernel equation forces the coordinate at its neighbor to vanish, contradicting a positive diagonal lower bound there.
\end{proof}
The inequalities must be applied in both directions, matching characteristic polynomials rather than assuming equal support graphs. If $r_{\min}$ is a lower bound for the isolate count on the opposite side, summing the projection inequality over a component of size $v$ gives nullity at least $\lceil vr_{\min}/15\rceil$. Iterating this restriction with the determinant bounds often removes all component configurations.

\begin{lemma}[Mixed-color rank signatures]\label{lem:mixed-rank}
Suppose both color partitions are $(14,1)$ and both majority supports have isolates. On the row side let $r$ be the isolate count, $k$ the nullity of the nonisolated majority block after subtracting $15I$, and $t$ the number of coordinates on which its kernel can be nonzero. Use $R,K,T$ for the corresponding column-side quantities. Then
\begin{equation}\label{eq:mixed-signature}
 r+k=R+K,\qquad r(15-R-T)\le15,\qquad
 r(15-R)\le15(K+1),
\end{equation}
and the two inequalities with the sides exchanged also hold.
\end{lemma}
\begin{proof}
In the kernel of $G-15I$, an isolated-majority row forces the minority coordinate to be zero. The nonisolated coordinates then belong to a $k$-dimensional kernel, and the minority equation imposes one nontrivial relation on those coordinates and the $r$ isolate coordinates. Thus the nullity is $r+k-1$. Matching nullities gives the first equation.

Let $Y$ consist of the $r$ isolated-majority rows of $H$. Then $YY^*=15I_r$ and $Y(K-15I)$ has rank at most one. Modulo one common row direction, its rows lie in $\ker(K-15I)$. That kernel is zero outside the $R$ isolate coordinates and the $T$ kernel-support coordinates. Restricting $Y$ to the other $15-R-T$ columns therefore gives rank at most one. Its squared Frobenius norm is $r(15-R-T)$, while its squared operator norm is at most 15. This proves the second inequality. After deleting only the $R$ isolate columns, the rank is at most $K+1$, proving the third in the same way.
\end{proof}

\begin{theorem}[Spectral rectangle inequality]\label{thm:rectangle}
Suppose $HH^*=nI+aT$ and $H^*H=nI+aS$, where $T,S$ are Hermitian, $a>0$, and every entry of $H$ has modulus one. Let $I,J$ be unions of invariant coordinate blocks of $T,S$. If
\[
 p(x)=\gcd(\chi_{T_I}(x),\chi_{S_J}(x))
      =x^d+c_1x^{d-1}+\cdots,
\]
then
\begin{equation}\label{eq:rectangle}
 |I||J|\le nd-ac_1.
\end{equation}
For $d=0$, the right side is interpreted as zero.
\end{theorem}
\begin{proof}
Put $X=H_{I,J}$. Invariance gives $T_IX=XS_J$, so the spectral decomposition of $X$ has nonzero blocks only at common eigenvalues. At a common eigenvalue $\lambda$, the rank is at most the smaller multiplicity and every squared singular value is at most $n+a\lambda$, because $XX^*\preceq nI+aT_I$. Summing these bounds with minimum multiplicities gives $nd-ac_1$. On the other hand, $\|X\|_F^2=|I||J|$.
\end{proof}
This test uses the trace of the common spectral part, not merely its dimension or the largest eigenvalue of the whole block. In our application $n=15$ and $a=3$.

\section{Finite verification and completion of the proof}\label{sec:finite}
\subsection{A common enumeration procedure}\label{sec:enumeration}
Each finite test follows the same sequence: fix an energy and a permitted color partition; assemble the possible majority support components; label their edges from \eqref{eq:entry-catalogue}; discard components using \cref{sec:moments}; and apply the Schur, arithmetic, and two-sided tests to the remaining assemblies. The exact determinant test is always $D>B$, not $D\ge B$.

Here is why the support and phase enumerations are complete. A connected support of order $v$ contains a spanning tree. List all nonisomorphic trees on $v$ vertices, add every subset of missing edges allowed by the energy budget, and retain every resulting graph up to verified isomorphism. Small graphs may instead be read from the complete graph atlas through seven vertices. For each graph, list every allowed norm partition, every placement of those norms on edges, and every associate class of a norm. Fix phases on a spanning tree by diagonal unit switching and list all six residual unit choices on each chord. This switching is used only for spectral and determinant necessary conditions; it is not asserted to be third-root equivalence of the original matrix. When a factorization test uses actual third-root rows, their original row cosets must be retained.

Component records contain their characteristic polynomial, determinant, and the minimum diagonal entry of the zero-eigenvalue projection. For a fixed polynomial, retaining the largest observed minimum projection diagonal is a safe relaxation, not a claim that the components are equivalent. Multisets of component records are assembled with all possible isolate counts. A leaf-containing component is omitted only after \cref{lem:projection} applies, or when the remaining vertices cannot be covered within the energy budget. In the latter case the minimum additional cost is one unit per two vertices.

The resulting computations are summarized below by input family, residual family, and exclusion. Their implementations are identified in \cref{app:replay}. The catalogue-generation arguments are part of the proof; a stored status word is not a substitute for them.

\subsection{Eliminating every size-13 majority block}
\begin{lemma}\label{lem:size13}
No strict improvement has color type $(13,2)$ or $(13,1,1)$ on either side.
\end{lemma}
\begin{proof}
By \cref{thm:orthogonality}, a 13-row same-color block has at least four nonzero internal pairs, hence internal energy at least 36. Its cross energy is at least 78. Therefore $Q\ge114$. For each permitted $Q\ge114$, complete component assembly with \eqref{eq:schur13} leaves no upper bound greater than $B$; these are the size-13 entries of \file{proof/new/exploration.json}, regenerated by \file{proof/new/explore.py}. The component costs satisfy $\sum v_i=13$, $9\sum u_i+c+h=Q$, $c\ge78$, and $h$ comes from the entry catalogue, so this is a finite enumeration with the coverage described above.
\end{proof}
After \cref{lem:size13}, a multiple-of-nine shell is pure-color on \emph{both} sides, while a shell congruent to six modulo nine has type $(14,1)$ on both sides. This follows from the cross-pair counts in \cref{lem:colors,lem:partitions}; it does not require choosing an orientation.

\subsection{The mixed-color shells}
The complete one-hub Schur bounds dispose of six shells directly. For reference, \cref{tab:schur} also records the two-hub bounds; all displayed decimals are rounded \emph{up} from exact rational upper bounds.
\begin{table}[ht]
\centering\caption{Upper bounds divided by $B$ from complete component assembly.}\label{tab:schur}
\begin{tabular}{@{}rcc@{}}\toprule
$Q$ & Size-13 majority & Size-14 majority\\\midrule
123&0.965877&0.984600\\
132&0.943735&0.979426\\
141&0.922622&0.971828\\
150&0.896803&0.963221\\
159&0.877308&0.939259\\
168&0.856213&0.919227\\
\bottomrule\end{tabular}
\end{table}

At $Q=87,96,105$, a majority support without isolates needs at least seven edges and hence $Q\ge7\cdot9+14\cdot3=105$. The no-isolate $Q=105$ case has exactly seven disjoint norm-9 edges and minimum cross norms. Equation \eqref{eq:schur14} gives
\[
 \det G\le216^7\left(15-\frac{42}{18}\right)
          =216^7\frac{38}{3}=B.
\]
Thus every strict improvement in these three shells has isolates on both sides. The complete rank-signature enumeration gives the following results.
\begin{table}[ht]
\centering\caption{Residual majority supports after \eqref{eq:mixed-signature}.}\label{tab:mixed}
\begin{tabular}{@{}rll@{}}\toprule
$Q$ & Residual support & Independence number\\\midrule
87 & none & ---\\
96 & $P_3\sqcup C_4\sqcup7K_1$ &11\\
   & $C_4\sqcup10K_1$ &12\\
   & $K_{2,3}\sqcup9K_1$ &12\\
105 & none & ---\\\bottomrule
\end{tabular}
\end{table}
The remaining three supports at $Q=96$ contradict \cref{thm:orthogonality}. The rank catalogue is generated before pairing signatures, using exact characteristic polynomials and kernel-support tests. Its entries retain a lower bound on independence number. The replay checks all compatible row/column signature pairs in both directions. An additional phase/cross enumeration retains 52 Gram candidates at $Q=96$; the sharp-reduction verifier exhibits ten independent vertices in each.

For $Q=114$, component bounds leave eight size-14 configurations. Exact core-phase and cross-vector enumeration leaves just one norm-compatible abstract Gram above $B$, of determinant $278010073125000000$. Its majority support is $K_4\sqcup10K_1$. It has an independent set of size 11, so it cannot be a third-root Gram. The final cross-vector step examines 3024 states; the complete core screening and the canonical residual Gram are regenerated by \file{proof/new/refine_q114.py}, \file{proof/new/cross_exact.py}, and \file{proof/new/finalize_high.py}.

\subsection{The low pure-color shells \texorpdfstring{$54,63,72,81$}{54,63,72,81}}
For $Q=54$ or 63 there are too few edges to cover all 15 vertices, so both supports have isolates. At $Q=72$, the only no-isolate possibility is $P_3\sqcup6K_2$ with unit normalized edge norms. Its Gram determinant is $3105\cdot216^6$, not an Eisenstein norm. Thus \cref{lem:projection} applies in all three shells, and every active component has minimum degree at least two. Such a component has at most eight edges and vertices. The graph atlas through seven vertices together with $C_8$ is complete in this range. All admissible weight and chord-phase assignments are included.

\begin{table}[ht]
\centering\caption{Exact isolate-case enumeration in the low pure-color shells. The assembly count is after requiring positive kernel projection on every active coordinate.}\label{tab:low}
\begin{tabular}{@{}rrrrr@{}}\toprule
$Q$ & Phase assignments & Assemblies & Norm-compatible $D>B$ & After projection\\\midrule
54&372&7&1&0\\
63&1086&15&1&0\\
72&7524&58&1&0\\
81&42720&285&22&0\\\bottomrule
\end{tabular}
\end{table}
At $Q=81$, a no-isolate support has nine unit-norm edges. It is either a forest with six components, each of order at most five, or $K_3\sqcup6K_2$. The forest cases fail the determinant and norm tests. The four residual triangle spectra have no zero eigenvalue after subtracting $15I$, so a no-isolate Gram cannot be paired with an isolate Gram. For the isolate case, the catalogue extends through nine energy units. The additional eight-vertex, nine-edge minimum-degree-two graphs are the figure-eight, barbell, and theta graphs; $C_8$ and $C_9$ are also included. The exact counts are in \cref{tab:low}.

The no-isolate triangle cases are excluded in \cref{app:q81}. That argument uses the same intertwining identity as the other shells; its last step is an elementary obstruction to a sum of two scaled permutation blocks. It is included in full because a spectral catalogue alone does not exclude the negative-flux triangle.

\subsection{The pure-color shells \texorpdfstring{$90,99,108$}{90,99,108}}
The complete component catalogues extend to ten energy units for $Q=90$ and to twelve units for $Q=99,108$. Large supports are generated by the spanning-tree procedure in \cref{sec:enumeration}, with the no-isolate covering budget and the isolate/leaf restriction applied separately. Table~\ref{tab:pure} records the assembly counts at the last stage; these are not counts of inequivalent original matrices.
\begin{table}[ht]
\centering\caption{Pure-color residual computations.}\label{tab:pure}
\begin{tabular}{@{}rccp{66mm}@{}}\toprule
$Q$ & Isolate assemblies & No-isolate assemblies & Remaining test\\\midrule
90 &807&74&Two no-isolate spectra; spectral rectangle.\\
99 &435&632&14 component--spectrum assemblies, 16 equal-spectrum pairs; spectral rectangle.\\
108&146&5570&No no-isolate norm survivor; eight isolate norm survivors all fail projection.\\\bottomrule
\end{tabular}
\end{table}

At $Q=90$, each surviving no-isolate spectrum consists of five $K_2$ components and one component of order five. Take the five-vertex component on one side and the ten pair coordinates on the other. The rectangle has squared Frobenius norm 50. Its common characteristic factor is respectively $x-1$ or $x^2-1$, giving the impossible bounds $50\le18$ or $50\le30$ by \cref{thm:rectangle}.

At $Q=99$, all isolate cases fail the two-sided projection test. The 14 no-isolate component--spectrum assemblies give 16 unordered pairs with equal total characteristic polynomial. For each pair, all nonempty component unions are examined. Every pair violates \eqref{eq:rectangle}. In particular, the three pairs that survive a weaker rank-times-largest-eigenvalue test give
\[
 40\le30,\qquad20\le18,\qquad36\le30.
\]
The certificate stores the two selected unions, their polynomials, their gcd, and the two integers in the failed inequality. These witnesses are independently re-evaluated by the sharp-reduction verifier.

At $Q=108$, neither an isolate nor a no-isolate candidate survives. The largest catalogue used for this step contains 15,899,506 normalized phase states. It is regenerated, not accepted from its reported size. The size-13 case is excluded by \cref{lem:size13}.

\subsection{The remaining pure-color shells}
For $Q=135,144,153,162$, complete component bounds followed by the two-sided projection iteration leave no configuration. At $Q=117$ the iteration leaves
\[
 6K_1+\text{a 9-vertex core},\quad
 7K_1+\text{an 8-vertex core},\quad
 8K_1+\text{a 7-vertex core},
\]
all with core energy 13 units. The respective rank bounds for $T$ are 5, 4, and 4. At $Q=126$ only seven isolates and an 8-vertex core of energy 14 units remain, with rank at most four. These assertions come from component-cost enumeration and the nullity lower bound in \cref{lem:projection}, not from an assumption about the equality structure.

The characteristic polynomial of an integral Hermitian Eisenstein matrix has integer coefficients: its coefficients lie in $\Eis$ and are real. If $\rank T\le d$, it follows that
\begin{equation}\label{eq:lattice}
 \det G=15^{15-d}3^dR,\qquad R\in\mathbb Z.
\end{equation}
Extra zero roots may be included when the actual rank is smaller. At $Q=126$, the rank-four envelope and $D>B$ give $397\le R\le401$. Every integer $15^{11}3^4R$ in this interval fails the Eisenstein norm test. The test is on the \emph{full determinant}; it is not a claim that none of the integers $R$ is a norm.

At $Q=117$, $\tr T^2=26$. In the rank-four case the only norm-compatible determinant quotient is $R=405$. The nonzero factor, padded by zero roots when necessary, has the form
\[
 x^4-13x^2-e_3x+e_4,\qquad
 |e_3|\le22,\qquad e_4=405-300-5e_3.
\]
The triangle bound gives the stated bound on $e_3$. None of these 45 polynomials has four real roots. In the rank-five case, the moment refinement leaves $R\in\{1983,1984,1987\}$, and the factor is
\[
 x^5-13x^3-e_3x^2+e_4x-e_5,
\]
where
\[
 |e_3|\le12,\qquad -85\le e_4\le50,\qquad |e_5|\le62,
 \qquad R=1500+25e_3+5e_4+e_5.
\]
Here the triangle bound is six, so $|e_3|\le12$. Newton identities give $e_4=(26^2-2\tr T^4)/8$, and the second moment bounds its range. The product bound $|e_5|\le(26/5)^{5/2}<62$ gives the last coefficient bound. Exactly 240 coefficient tuples meet the displayed integer restrictions; none has five real roots. The verifier uses rational real-root isolation, so failure of real-rootedness is an exact contradiction to Hermitian symmetry.

\subsection{Assembly of the global argument}
\begin{proof}[Proof of \cref{thm:main}]
Suppose $\det(HH^*)>B$. By \cref{cor:shells}, its energy lies in the 23-element set \eqref{eq:qset}. Lemmas~\ref{lem:partitions} and \ref{lem:size13} reduce the color types to pure color on both sides or $(14,1)$ on both sides. The mixed-color shells are excluded by \cref{tab:schur,tab:mixed} and the $Q=114$ independent-set obstruction. The pure-color shells are excluded by the low-energy catalogues, \cref{app:q81}, the spectral rectangles, and the rank/moment computations above. These cases partition \eqref{eq:qset}. Thus no strict improvement exists. The explicit matrix in \cref{app:matrix} has squared determinant $B$, as verified independently in \cref{sec:witness}, proving equality.
\end{proof}

\section{The displayed extremizer}\label{sec:witness}
\subsection{Normalization, Gram matrix, and orthogonal rows}
Let $H_d=(\omega^{e_{ij}})$ be the dephased matrix in \cref{app:matrix}; its first row and first column are all ones. Dephasing changes the determinant by a unit, not its absolute value. Exact Eisenstein Bareiss elimination gives
\begin{equation}\label{eq:dephased-det}
 \det H_d=-362797056-604661760\omega,
 \qquad \N(\det H_d)=B.
\end{equation}
The same value is obtained by computing the determinant of its Gram matrix.

After the explicit row switching and ordering in \cref{app:matrix}, the Gram becomes
\begin{equation}\label{eq:canonical-gram}
 G_c=
 \begin{pmatrix}
 I_7\otimes C&\delta\one_{14}\\
 \bar\delta\one_{14}^T&15
 \end{pmatrix},\qquad
 C=\begin{pmatrix}15&3\\3&15\end{pmatrix}.
\end{equation}
Thus its support consists of seven triangles sharing one vertex. There are seven off-diagonal pairs of norm 9, fourteen of norm 3, and 84 zero pairs; hence $Q=105$. Its spectrum is
\begin{equation}\label{eq:witness-spectrum}
 12\ (7\text{ times}),\qquad18\ (6\text{ times}),\qquad
 \frac{33+\sqrt{177}}2,\quad\frac{33-\sqrt{177}}2.
\end{equation}
Indeed, pair differences give the eigenvalue 12; the six-dimensional zero-sum space of pair sums gives 18; and the remaining two-dimensional block has trace 33 and determinant 228. Consequently,
\[
 \chi_{G_c}(x)=(x-12)^7(x-18)^6(x^2-33x+228),\qquad
 \det G_c=12^7\,18^6\,228=B.
\]
This is a short algebraic verification of the benchmark determinant, independent of the elimination calculation.

The largest mutually orthogonal subset of rows of this particular matrix has size seven. It consists of one endpoint from each matching edge, never the common vertex. There are exactly $2^7=128$ such subsets. Each noncentral row is orthogonal to 12 other rows, whereas the central row is orthogonal to none. These are different statistics from the global orthogonality capacity $M_3(15)=9$. The column Gram has the same support and norm statistics, and the largest orthogonal column subsets likewise have size seven. The dephased row and column Grams are not equal; dephasing does not imply normality.

\subsection{Automorphisms}
We use the following fixed convention:
\[
 \Aut_{\muT}(H)=\{(M,N):M,N\text{ are monomial with nonzero entries in }\muT,
                         \ MHN^*=H\}.
\]
It contains the central subgroup $\{(\omega^jI,\omega^jI):j=0,1,2\}$, acting trivially on entries. The quotient by this subgroup is called the projective monomial automorphism group here. Complex conjugation and transposition are not included.

\begin{proposition}\label{prop:automorphisms}
For the displayed extremizer,
\[
 |\Aut_{\muT}(H_d)|=1008,
 \qquad \Aut_{\muT}(H_d)\cong C_6\times\GL(3,2).
\]
Its projective monomial automorphism group has order 336 and is isomorphic to $C_2\times\GL(3,2)$.
\end{proposition}
\begin{proof}
Use the equivalent representative $H_c$ whose Gram is \eqref{eq:canonical-gram} and whose last row is all ones; its exponent matrix and the switching are included in the data. Every row-monomial Gram stabilizer fixes the unique degree-14 vertex and permutes the seven pairs, possibly interchanging their endpoints. Since all hub entries equal $\delta$, its row phases must be equal. Modulo the central scalar subgroup, the row action is therefore one of exactly $7!2^7=645120$ permutations.

For such a permutation, the last row of $H_c$ forces all compensating column phases to be one. A compensating column permutation exists precisely when the permuted columns equal the original column set. The columns are distinct because $H_c$ is nonsingular, so it is unique. Exhaustive integer comparison of these 645120 row permutations gives exactly 336 pairs. The program verifies all 225 entries for every retained pair.

The resulting permutation group has center of order two, generated by interchanging the endpoints of all seven pairs, and derived subgroup of order 168. Their intersection is trivial. On the seven pairs, the derived subgroup acts faithfully and preserves the seven triples
\[
 123,\quad145,\quad167,\quad246,\quad257,\quad347,\quad356.
\]
Every pair of points lies in exactly one triple, so this is the Fano plane. The program independently checks all $7!$ permutations of the seven points and finds 168 preserving these triples. Its automorphism group is $\GL(3,2)$: label the points by the nonzero vectors of $\mathbb F_2^3$, with lines the triples summing to zero; a line-preserving permutation is determined by the images of a basis and preserves addition. Thus the projective group is $C_2\times\GL(3,2)$. The representatives just obtained are ordinary permutation pairs, so the scalar $C_3$ splits off in the full monomial group. This gives $C_3\times C_2\times\GL(3,2)$, as asserted.
\end{proof}
The much larger monomial stabilizer of $G_c$ alone must not be confused with the stabilizer of $H_c$. Most permutations of its seven pairs do not preserve the column set of a factor. Likewise, determining the automorphism group of one maximizer does not establish uniqueness of its Hadamard equivalence class. The search above discards $D=B$ when proving the upper bound and has not classified all equality cases.

\section{Extensions to other orders and alphabets}\label{sec:generalization}
The determinant envelope, Schur bounds, projection arguments, and spectral rectangle inequality apply to arbitrary complex unit-modulus matrices, with the order and scaling changed appropriately. Finite Gram search also extends to any fixed finite alphabet: the set of sums of $n$ alphabet products is finite. This remains true even when the full cyclotomic integer ring is not a discrete subset of the chosen complex embedding.

The arithmetic and coding parts require more care. For a prime alphabet, \cref{prop:prime-code} turns orthogonality into a balanced-difference condition, so any upper bound on the corresponding difference-matrix parameter is an obstruction to partial or full Butson matrices. A bound below $n$ proves nonexistence of $BH(n,p)$; it can also remove sparse Gram supports through their independence number. For a composite alphabet the single-character implication fails: $(1,1)$ and $(1,-1)$ are orthogonal fourth-root rows, but their differences are not uniform over $\mathbb Z_4$. Difference matrices over general finite abelian groups are instead characterized by vanishing for \emph{all} nontrivial characters. Coding-theoretic methods are therefore not confined to prime groups, although the particular equivalence used here is.

The order-15 result does not give a uniform gap from Hadamard's bound in an infinite congruence class. A useful quantitative target follows from the same envelope.
\begin{corollary}\label{cor:gap}
Let $H_n$ be square unit-modulus matrices and let $Q_n$ be their off-diagonal Gram energies. If $Q_n\ge cn^2$ for a fixed $c>0$, then
\[
 \limsup_{n\to\infty}\frac{|\det H_n|}{n^{n/2}}
 \le\left[(1+\sqrt{2c})e^{-\sqrt{2c}}\right]^{1/2}<1.
\]
\end{corollary}
\begin{proof}
Apply \cref{thm:envelope} with $m=\mu=n$ and use monotonicity in energy. The lower energy $cn^2$ corresponds to
$t_n=\sqrt{2cn/(n-1)}\to\sqrt{2c}$. After division by $n^n$, the squared-determinant bound tends to $(1+\sqrt{2c})e^{-\sqrt{2c}}$.
\end{proof}
Nonexistence of a Hadamard matrix only says that the energy is nonzero. Even a linear lower bound on the number of nonorthogonal pairs need not give the quadratic energy lower bound in \cref{cor:gap}. Establishing such a bound in a suitable infinite family is a distinct problem; no such claim is made here.

\section{Reproducibility and scope}\label{sec:repro}
The verification programs and exact data are available in the companion repository at
\url{https://github.com/ddy314/order15-mu3-maxdet}.
The programs under \file{proof/} generate the finite catalogues and verify the 38 coarse energy shells. The program \file{tools/verify_sharp_reduction.py} checks the reduction to 23 energies and the corresponding exclusion witnesses. Exact matrix data are stored under \file{data/}; \file{tools/verify_structure.py} verifies the extremizer and its automorphism group.

The complete verification is run from the repository root with
\begin{verbatim}
python -m pip install -r proof/requirements.txt
python verify_all.py
\end{verbatim}
The three stages may also be run separately:
\begin{verbatim}
python proof/run_final.py
python tools/verify_sharp_reduction.py
python tools/verify_structure.py
\end{verbatim}
Python assertions must remain enabled. The last command requires a C++17 compiler for the exhaustive automorphism enumeration.

The 38-shell verification includes two complete normalized-column scans for each of 194 Gram candidates. Each scan examines $3^{14}=4{,}782{,}969$ columns. These scans provide additional checks: the 23-shell proof excludes all 194 candidates by the orthogonality bound, using explicit independent sets of at least ten vertices. The sharp-reduction verifier reconstructs each Gram, checks its determinant, and supplies the independent-set witness.

The analytic and algebraic arguments are proved in the text. The finite support, norm, phase, and coefficient families are generated and checked by the repository programs. The results $B_3(15,10)=12$ and $M_3(15)=9$ are published external classifications; they are cited rather than regenerated here. This is a computer-assisted mathematical proof, not a proof-assistant formalization. It neither classifies all maximizing equivalence classes nor establishes the asymptotic hypothesis of \cref{cor:gap}.

\section*{Acknowledgment}
The author thanks P\'adraig \'O Cath\'ain for detailed comments on the manuscript and for emphasizing the connection with the real Gram-matrix search method and the structure of the extremizer.

\section*{AI use declaration}
AI tools contributed to the research reported in this paper; the author takes full responsibility for all of its content.

\appendix
\section{An explicit dephased matrix and its Gram}\label{app:matrix}
The following matrix $E_d$ encodes $H_d$ by $(H_d)_{ij}=\omega^{(E_d)_{ij}}$. Its first row and first column are zero.
\[
\setlength{\arraycolsep}{3.5pt}
\begin{pmatrix}
0&0&0&0&0&0&0&0&0&0&0&0&0&0&0\\
0&1&0&0&0&2&2&0&2&1&2&1&1&1&2\\
0&1&1&2&2&1&0&2&0&1&0&1&0&2&2\\
0&2&1&2&0&0&2&1&1&2&0&2&0&1&1\\
0&2&0&0&1&1&1&2&0&2&2&2&1&0&1\\
0&2&2&2&1&1&1&0&2&0&0&1&2&1&0\\
0&1&2&0&2&1&0&1&1&0&2&2&2&1&0\\
0&1&1&1&1&2&2&2&2&0&1&2&0&0&0\\
0&1&1&0&2&0&1&0&2&2&1&0&2&2&1\\
0&0&2&1&1&0&1&0&1&1&2&2&0&2&2\\
0&0&2&1&2&2&0&1&2&2&0&1&1&0&1\\
0&2&0&2&1&2&0&1&1&1&1&0&2&0&2\\
0&0&2&1&0&1&2&2&0&1&1&0&2&1&0\\
0&2&0&1&2&0&2&2&1&0&1&1&1&2&0\\
0&0&1&2&0&2&1&1&0&0&2&0&1&2&0
\end{pmatrix}

\]
Let
\begin{align*}
 \pi&=(1,15,2,4,3,5,6,7,8,14,9,12,10,11,13),\\
 a&=(0,1,1,2,2,0,0,0,2,1,2,1,0,0,0),\\
 b&=(0,0,2,1,0,1,2,2,0,1,1,0,2,1,0).
\end{align*}
Define $(H_c)_{ij}=\omega^{-a_{\pi_i}}(H_d)_{\pi_i,j}\omega^{-b_j}$. Its last row is all ones, its row Gram is exactly $G_c$ in \eqref{eq:canonical-gram}, and
\[
 \det H_c=241864704-362797056\omega.
\]
For an explicit expression for the Gram of the \emph{displayed dephased} matrix itself, let $P_{i,\pi_i}=1$ and $D=\diag(\omega^{-a_1},\ldots,\omega^{-a_{15}})$. Then
\begin{equation}\label{eq:dephased-gram}
 H_dH_d^*=D^*P^TG_cPD.
\end{equation}
Equations~\eqref{eq:canonical-gram} and \eqref{eq:dephased-gram}, with the explicit $P,D$ above, specify every Gram entry. The exact verification tool also writes both full Grams as integer pairs $[a,b]$ representing $a+b\omega$.

In the dephased row ordering, row 13 is the hub and the norm-9 matching pairs are
\[
 (1,15),\ (2,4),\ (3,5),\ (6,7),\ (8,14),\ (9,12),\ (10,11).
\]
Choosing one row from each pair gives every maximum orthogonal row subset.

\section{The triangle boundary at \texorpdfstring{$Q=81$}{Q=81}}\label{app:q81}
We give the decomposition argument for the no-isolate support $K_3\sqcup6K_2$ left in \cref{sec:finite}. The normalized triangle flux is a sixth root. The nonreal-flux triangle polynomial is coprime to $x^2-1$, so a triangle-to-pair rectangle would vanish by intertwining. Positive real flux gives normalized triangle spectrum $\{2,-1,-1\}$; against all twelve pair coordinates, \cref{thm:rectangle} gives $36\le24$. Only negative real flux remains. Its Gram spectrum is $\{18,18,9\}$, and both row and column Grams must have this case because their total spectra agree.

The eigenspace at 9 is supported on the three triangle coordinates on both sides. Its intertwiner has squared singular value 9. The corresponding $3\times3$ submatrix of $H$ has squared Frobenius norm 9, so it has rank one. Since its entries are third roots, allowed row and column phase changes make this submatrix $J_3$. The two triangle Gram blocks then become $18I_3-3J_3$: their eigenvector at 9 is $\one_3$, forcing both off-diagonal entries in each row to be $-3$.

Each pair Gram can be gauged by third-root phases to an off-diagonal entry $3$ or $-3$. A negative pair has a symmetric eigenvector at 12, an eigenvalue absent from the triangle block. Intertwining would require two third-root entries in each triangle column to sum to zero. This is impossible. Hence all six pair blocks are positive.

In pair-symmetric and pair-antisymmetric coordinates, each pair-to-pair block is
\[
 \begin{pmatrix}a&b\\b&a\end{pmatrix},\qquad a,b\in\muT.
\]
Write the resulting $6\times6$ blocks as $S=(a+b)$ and $D=(a-b)$. Let $X$ be the $6\times3$ matrix of pair-to-triangle rows, one row from each equal pair, and let $Y$ be the analogous $3\times6$ triangle-to-pair matrix. In orthonormal pair coordinates, the remaining block is
\[
 \begin{pmatrix}S&\sqrt2X\\\sqrt2Y&J_3\end{pmatrix},
 \qquad DD^*=D^*D=12I_6.
\]
Every row and column of $D$ has four nonzero entries of norm 3 and two zero entries. At a zero of $D$, the corresponding entry of $S$ is twice a third root; otherwise it is minus a third root.

The Gram equations give
\[
 X^*X=YY^*=9I_3-3J_3,\qquad X\one_3=0,\quad Y^*\one_3=0.
\]
A balanced third-root triple is, up to a third-root scalar, either
$u=(1,\omega,\omega^2)$ or $\bar u$. These two triples are orthogonal. The displayed Gram equations therefore force exactly three of each orbit among the six rows of $X$, and similarly among the six columns of $Y$. Independent pair phases and permutations put these into two groups of three identical triples.

The cross-Gram equations yield $SY^*=0$ and $S^*X=0$. Thus every row and column of $S$ sums to zero within each group of three. A group cannot contain two zeros of $D$, since $2a+2b-c=0$ has no solution in third roots. Each group therefore contains exactly one zero, and $2a-b-c=0$ forces $a=b=c$. It follows, using both row and column conditions, that each $3\times3$ block of $S$ has the form
\[
 \alpha(3P-J_3),\qquad \alpha\in\muT,\quad P\text{ a permutation matrix}.
\]
The two row groups of $X$ are orthogonal, so the off-diagonal block of $SS^*$ is zero. This would require
\[
 \alpha(9R-3J_3)+\beta(9T-3J_3)=0,
 \qquad\alpha,\beta\in\muT,\quad R,T\text{ permutation matrices}.
\]
If $R\ne T$, an entry where one is 1 and the other is 0 forces $6\alpha=3\beta$ or its reverse, contradicting equal moduli. If $R=T$, it forces $\alpha+\beta=0$, impossible for third roots. This contradiction completes the boundary case. The verification program also checks all 324 phase/permutation possibilities.

\clearpage
\section{Proof obligations and executable checks}\label{app:replay}
The following table identifies the generators and checks used in the proof. Paths are relative to \file{proof/}, except for the explicitly marked top-level tools. JSON outputs are reconstructed by the verification programs.
\small
\begin{longtable}{@{}>{\raggedright\arraybackslash}p{37mm}>{\raggedright\arraybackslash}p{107mm}@{}}
\toprule
Proof obligation & Executable source and finite coverage\\\midrule
\endfirsthead
\toprule Proof obligation & Executable source and finite coverage\\\midrule
\endhead
Arithmetic and color partitions & \file{order15_mu3_unified_audit/programs/verify_foundations.py}; all 136 entry triples, exact tail and partition inequalities.\\[3pt]
Low pure shells & \file{order15_mu3_unified_audit/programs/verify_low_energy.py}, \file{order15_mu3_unified_audit/programs/verify_q81.py}; graph-atlas supports, the explicitly generated larger low-energy graphs, every norm and chord phase, two-sided projection.\\[3pt]
Component maxima & \file{new/small_catalogue.py}, \file{new/bounds.py}; complete small support maxima, tree bounds, rational moment and eigenvalue bounds.\\[3pt]
Mixed-color bounds & \file{new/explore.py}; every component cost partition including isolates, one- and two-hub Schur bounds, cross and outside-pair energies.\\[3pt]
Mixed signatures & \file{new/rank_catalogue.py}, \file{new/mixed_rank_frontier.py}; exact nullity and kernel-support records, all compatible row/column signature pairs.\\[3pt]
$Q=114$ residual & \file{new/refine_q114.py}, \file{new/cross_exact.py}, \file{new/finalize_high.py}; eight configurations and the full core/cross search. The top-level sharp verifier supplies an independent-set witness for the final Gram.\\[3pt]
Pure component catalogues & \file{new/pure_catalogue.py}, \file{new/pure_extended.py}; spanning-tree plus extra-edge support generation through the required budgets, all norm placements and chord phases.\\[3pt]
$Q=90,108$ assemblies & \file{new/pure90_frontier.py}, \file{new/pure108_frontier.py}; both isolate regimes, exact norm and projection tests.\\[3pt]
$Q=99$ rectangles & \file{last_three/pure99_frontier.py}, \file{last_three/pure99_pairs.py}, \file{last_three/pure99_spectral.py}; 14 residual assemblies, 16 equal-spectrum pairs, all component unions.\\[3pt]
High pure shells & \file{new/explore_pure.py}, \file{new/projection_frontier.py}, \file{new/cap_exact.py}; component-cost partitions, capped moment extrema, iterative two-sided nullity restrictions.\\[3pt]
$Q=117,126$ arithmetic & \file{new/spectral_integer.py}, \file{new/finalize_high.py}; bounded integer characteristic coefficients, exact real-root isolation, full-determinant norm tests.\\[3pt]
Sharp reduction checks & top-level \file{tools/verify_sharp_reduction.py}; 38/29/23 energy sets, 194 explicit independent sets, $Q=114$, all residual rectangles, 285 real-root tests, five full-determinant norm obstructions.\\[3pt]
Witness and automorphisms & top-level \file{tools/verify_structure.py}, \file{tools/matrix_exact.py}, \file{tools/enumerate_aut.cpp}; exact dephasing and Gram arithmetic, Bareiss determinants, all 645120 row candidates, Fano action and group checks.\\\bottomrule
\end{longtable}

\end{document}